\documentclass[12pt, english]{amsart}
\usepackage[utf8]{inputenc}
\usepackage{xcolor}
\usepackage[margin=1in]{geometry} % comment for defaut margins

\usepackage{amsmath,amssymb,amsthm}
\usepackage{stmaryrd}
\usepackage{graphicx}
\usepackage{caption}
\usepackage{tikz}
\usepackage{tikz-cd}
\usetikzlibrary{decorations.markings}
\usepackage{hyperref}
\usepackage{url}
\usepackage{accents}
\usepackage{mathrsfs}
\usepackage[]{tikz-cd}
\usepackage[all]{xy}
\usepackage{multicol}
\usepackage{mathtools}
\usepackage{color}
\usepackage[symbol]{footmisc}

\theoremstyle{plain}
\newtheorem{theorem}{Theorem}[section]

\newtheorem{lemma}[theorem]{Lemma}

\newtheorem{proposition}[theorem]{Proposition}

\newtheorem*{thm*}{Theorem}

\theoremstyle{definition}

\newtheorem{remark}[theorem]{Remark}

\newtheorem{definition}[theorem]{Definition}

\newtheorem{example}[theorem]{Example}

\newcommand{\R}{\mathbb{R}}
\newcommand{\Z}{\mathbb{Z}}

\newcommand{\M}{\mathbb{M}}
\newcommand{\A}{\mathbb{A}}
\renewcommand{\H}{\mathbb{H}}
\renewcommand{\S}{\mathcal{S}}
\newcommand{\ip}[2]{\langle {#1}, {#2} \rangle}
\newcommand{\Pro}{\mathbb{P}}
\newcommand{\hermdecomp}[1]{\mathbb{#1}^{M\times N}_{\mathrm{sym}} \times \mathbb{#1}^{M\times N}_{\mathrm{skew}}}
\newcommand{\C}{\mathbb{C}}

\newcommand{\K}{\mathbb{K}}
\renewcommand{\o}[1]{\overline{#1}}
\renewcommand{\b}{\mathcal{B}}

\newcommand{\V}{\mathbb{V}}
\newcommand{\B}{\mathcal{B}}
\newcommand{\UU}{\mathcal{U}}
\newcommand{\VV}{\mathcal{V}}
\renewcommand{\L}{\mathbb{L}}
\DeclareMathOperator{\supp}{Supp}

\DeclareMathOperator{\GL}{GL}

\DeclareMathOperator{\rank}{rank}
\title{The generic crystallographic phase retrieval problem} 
\author{Dan Edidin, Arun Suresh}
\address{Department of Mathematics, Universit of Missouri, Columbia, MO 65211}
\email{edidind@missouri.edu, aszxy@missouri.edu}
\date{\today}
\begin{document}
\maketitle
\begin{abstract}
  In this paper we consider the problem of recovering a signal $x \in
  \R^N$ from its power spectrum assuming that the signal is sparse
  with respect to a generic basis for $\R^N$. Our main result is that
  if the sparsity level is at most $\sim\! N/2$ in this basis then the
  generic sparse vector is uniquely determined up to sign from its
  power spectrum. We also prove that if the sparsity level is $\sim\!
  N/4$ then every sparse vector is determined up to sign from its
  power spectrum. Analogous results are also obtained for the power spectrum of
  a vector in $\C^N$
which extend earlier results of Wang and Xu \cite{wang2014phase}.
\end{abstract}

\section{Introduction}
The crystallogrpahic phase retrieval problem entails recovering a sparse
signal $x \in \R^N$ from its power spectrum which denote $|Fx|^2$. It is, by definition, the componentwise magnitude-squared of the discrete Fourier transform $Fx$ of the signal. Due to its connections to X-ray crystallography
there is an extensive literature,
particularly for binary signals,
going back to the work of Patterson \cite{patterson1934fourier, patterson1944ambiguities}.
The mathematical foundations of this problem for vectors whose non-zero entries are generic was studied in
\cite{bendory2020toward, ghosh2022sparse}. In \cite{bendory2020toward} several conjectures were made regarding the level of sparsity required to recover a generic signal up to reflection and shift ambiguities from its power spectrum. 

The purpose of this paper is to consider a relaxation
which we call the {\em generic crystallographic phase retrieval problem}.
In our setup we consider signals which have sparse representations
with respect to a {\em generic basis} for $\R^N$. In doing so we obtain optimal provable bounds
for the sparsity level required to solve this phase retrieval problem.

Given a signal $x_0 \in \R^N$  whose expansion in terms
of a basis $\V$ is sparse, 
the generic phase retrieval problem can be interpreted as follows:
Determine whether the intersection 
of two  non-convex sets $\B_{x_0}  \cap \S_\V = \{\pm x_0\}$.
Here $\B_{x_0} =\{x\in \R^N| |Fx|^2= |Fx_0|^2\}$ is the set of vectors with the same power
spectrum as $x_0$ and $\S_\V$ is the set of vectors in $\R^N$ whose expansion
in the basis $\V$ has at most $M < N$ non-zero coordinates.
Geometrically we can view $\B_{x_0}$ as the orbit of $x_0$ under the product
of planar rotation groups of total dimension $\lceil N/2 \rceil -1$ \cite[Section 5]{bendory2020toward}, while  $\S_\V$ is the union of ${N \choose M}$ linear subspaces
of dimension $M$.  A simple dimension count shows that this intersection
should be positive-dimensional unless $M \leq \lfloor N/2 \rfloor$. Our main result
states that if $\V$ is a generic basis then the bound $M \leq \lfloor N/2 \rfloor$ is optimal for generic vectors $x_0$ whose expansion in the basis $\V$
has at most $M$ non-zero coordinates. Moreover,
we also prove that if $M < (N+3)/4$ then {\em every vector} which is $M$-sparse
with respect to $\V$ can be recovered from its power spectrum.

\subsection{Statement of the main result}
Let $\K = \R$ or $\K = \C$
and let $\V = \{v_1,\dots, v_{N}\}$ be an ordered basis of $\K^N$.
If $x \in \K^N$ is a vector which we expand in terms
of the basis $\V$ as $x = \sum_{n=1}^{N} x_n v_n$ then we define
the {\em support} of $x$ with respect to $\V$ to be the set
$\supp_{\V}(x) = \{n| x_n \neq 0\} \subset [1,N]$.
We say that $v$ is $M$-sparse (with respect to $\V$) if
$|\supp_{\V}(v)| \leq M$.
Given a subset $S \subseteq [1,N]$ with $|S|=M$, set
$$\mathbb{L}_S(\V) = \mathrm{span}(v_i| i \in S).$$
It is clear that $x \in \K^N$ is $M-$sparse iff $x \in \L_S(\V)$ for
some subset $S \subset [1,N]$ with $|S| = M$.

\begin{theorem}\label{thm.realsparse}
  Let  ${\V}$ be a generic basis for $\R^N$.
  \begin{itemize}
\item[(i)]  If  $M< N/4 +1$ ($N$ even) or $M < (N+3)/4$ ($N$ odd) then
every $M$-sparse vector can be recovered up to a global sign from its power spectrum.
\item[(ii)] If $M < \lfloor N/2 \rfloor +1 $ then a generic $M$-sparse vector can be recovered up to a global sign from its power spectrum.
  \end{itemize}
\end{theorem}

\begin{remark}
If $x=(x_1, \ldots , x_{N}) \in \R^N$ and $\sigma \in D_{2N}$ is an element of the dihedral group which acts on $\R^N$ by cyclic shifts
and reflection, then the
magnitudes of the Fourier coefficients of $x$ and $\sigma x$ are the same~\cite{beinert2015ambiguities}. In particular if a vector
is $M$-sparse with respect to the standard basis then any dihedral shift of the vector is another $M$-sparse vector with the same power spectrum. Thus, the inherent ambiguity group of the phase retrieval problem for sparse vectors is the dihedral group. However, if $x$ has a sparse expansion with respect to some other basis $\V$ then a dihedral shift of
the vector will not in general be sparse with respect to $\V$. For this reason the ambiguity group of
a vector which is sparse with respect to a generic basis will only be $\{\pm 1\}$. 
\begin{example}
Consider the orthonormal basis $\V = \{e_1, e_2, v_3, v_4\}$ for $\R^4$ where $e_1 = (1,0,0,0), e_2= (0,1,0,0), v_3 = (0,0,\frac{\sqrt{2}}{2}, \frac{\sqrt{2}}{2}), v_4 =(0,0,-\frac{\sqrt{2}}{2},\frac{\sqrt{2}}{2})$.
The vector $x = (1,1,0,0) = e_1 + e_2$ is 2-sparse with respect to this basis (i.e. only two of the coefficients in the expansion in terms of the basis $\V$ are non-zero). However the shifted vector $(0,1,1,0)= e_2 + {1\over{\sqrt{2}}}(-v_3 + v_4)$ is not $2$-sparse with respect to $\V$.
\end{example}
\end{remark}

\begin{remark} Even if we consider signals which are sparse with respect to the standard basis up to dihedral equivalence, Theorem~\ref{thm.realsparse}(i) does not hold. The reason is that there are well-known examples of
binary signals which are not dihedrally equivalent but which have the same power spectrum. As first observed by Patterson~\cite{patterson1944ambiguities}, two binary signals have the same power spectrum if and only if their support sets have the same cyclic difference sets. However, there are numerous examples of non-dihedrally equivalent subsets of $[1,N]$ which have the same cyclic difference sets. For more detail see~\cite{bendory2020toward}.\end{remark}

\subsubsection{Comparison of Theorem~\ref{thm.realsparse}(ii) with the conjectural bounds of \cite{bendory2020toward}}
In \cite[Conjecture 4.7]{bendory2020toward} it is conjectured
that for a generic signal $x \in \R^N$ supported on a set $S\subset [1,N]$
with respect to the standard basis, and $|S-S| > M$
then $x$ can be recovered up to the action of the dihedral group from its power spectrum. Here $S-S$ is the {\em cyclic difference set} of $S$, and a necessary condition for $|S-S| >  M$ is
that $M = |S| < \lfloor N/2 \rfloor +1$ which is the
same bound we obtain for generic bases. However, if
the support set $S$ is an arithmetic progression then
the generic vector with support $S$ is provably non-recoverable \cite[Proposition 4.4]{bendory2020toward}. This further shows
that the standard basis for $\R^N$ is not completely generic. By contrast
if we assume that our basis $\V = (v_1,\ldots , v_N)$
is generic and if $M < \lfloor N/2  \rfloor +1$ then for {\em every subset}
$S \subset [1,N]$ the generic vector in the subspace $\L_S(\V)$ spanned
by $\{v_i| i \in S\}$ can be recovered from its power spectrum.

\subsection{Physical motivation}
As the title suggests, this work is motivated by X-ray crystallography
which is a prevalent technology for determining three-dimensional
molecular structures.  In
X-ray crystallography, the signal is the electron density of the
crystal -- a periodic arrangement of a repeating compactly supported
unit which is illuminated with a beam of X-rays producing a diffraction
pattern. The diffraction pattern is equivalent to the magnitude of the Fourier transform
of the electron density function of the crystal \cite{elser2018benchmark}.

If we discretize then, due to the periodic nature of a
crystal, the electron density function can be viewed as a function on
a finite abelian group. When we take the group to be $\Z_N$ the
problem of recovering the electron density function from its
diffraction pattern is the problem of recovering a vector $x_0 \in
\R^N$ from its power spectrum \cite{bendory2023finite}.

This problem is ill-posed
without prior information on the signal. To that end mathematicians
have studied this problem,
under the
assumption that $x_0$ is sparse in the standard basis, motivated by
the fact that the density of atoms in a typical protein crystal is
estimated to be on the order of $1\%$
\cite{elser2017complexity}. However, there are also models which
encode the sparsity of a 3-D molecular stucture by a sparse expansion
in other bases such as wavelets \cite{bendory2023autocorrelation, daubechies1992ten}. 
This model
of sparsity closely reflects our notion of generic sparsity.  For the
mathematical problem of recovering molecular structures from the
second-moment of cryo-EM measurements this sparsity assumption was made in
\cite{bendory2022sample, bendory2023autocorrelation}. Indeed this previous
work on cryo-EM was one of our motivations for writing this paper.

\subsection{Fourier phase retrieval and frame phase retrieval.}
A related question which has been extensively studied in
the literature \cite{balan2006signal, balan2009painless, candes2013phase, candes2013phaselift,fickus2014phase, wang2014phase, bodmann2015stable, conca2015algebraic, botelho-andrade2016phase, bajarovska2016phase,
  iwen2017robust, li2019phase, evans2020conjugate} is the problem of recovering a signal $v \in \K^M$ from $N$-phaseless frame measurements $(|\langle v, \alpha_n\rangle|)_{n=1}^N$, where $\K = \R$ or $\K = \C$.
  The frame phase retrieval problem can be reinterpreted
  as the problem of recovering
  a vector $x_0 \in \K^N$ from $|x_0|$ with the constraint that $x_0$
  lies in the $M$-dimensional linear subspace of $\K^N$ which is the range of the analysis operator of the frame\footnote{Throughout this paper if 
$x \in \K^N$ then $|x|$ refers to the vector whose entries are the absolute values of the entries of $x$.}.

When $\K =\R$ this is a very different problem than
recovering a vector from its power spectrum. The reason is
that, for the frame phase retrieval problem, the
set $\B_{x_0} = \{y \in \R^N| |y| = |x_0|\}$ is zero-dimensional because
it consists of the $2^N$ vectors whose
coordinates differ from the corresponding coordinates of $x_0$ by a sign
In contrast, the
set of vectors with the same power spectrum as $x_0$ has dimension
$\sim\! N/2$. This is why the bound $M < (N+1)/2$ of \cite{balan2006signal}
for real frame phase retrieval is much stronger than the
bound $M < (N+3)/4$ obtained in 
Theorem \ref{thm.realsparse}.

On the other hand when $\K = \C$ and $x_0 \in \C^N$ then if $F$ is the discrete Fourier transform matrix, the
sets $\{y||y| = |x_0|\}$ and $\{y| |Fy| = |Fx_0|\}$ both have real dimension
$N$ because they are both orbits of $(S^1)^N$ actions. For the complex Fourier phase retrieval problem we obtain the following result with the same bounds as those obtained
for complex frame phase retrieval.
\begin{theorem} \label{thm.complexsparse}
  Let  ${\V}$ be a generic basis for $\C^N$.
  \begin{itemize}
\item[(i)]  If  $M < (N+3)/4$ then
  every $M$-sparse vector can be recovered up to a global phase from its power spectrum.
\item[(ii)] If $M < (N+1)/2$ then a generic $M$-sparse vector can be recovered up to a global phase from its power spectrum.
  \end{itemize}
  \end{theorem}
Indeed the proof of Theorem \ref{thm.complexsparse} is obtained by
reducing the problem to a complex frame phase retrieval problem
using the observation that multiplication by the complex
$N \times N$ DFT matrix induces an automorphism on the space of $N$-element
complex frames in $\C^M$.
We then
extend techniques developed in \cite{conca2015algebraic} to pairs
of overlapping frames to prove our result. We note that Theorem \ref{thm.complexsparse}(i)
is, via our reduction, equivalent to
\cite[Theorem 2.2]{wang2014phase}, although our proof is different.

By contrast the proof of Theorem \ref{thm.realsparse} cannot be
reduced to a frame phase retrieval problem because the DFT matrix has
complex entries.  To prove the theorem we realize the auto-correlation
as the {\em second moment} for the action of the dihedral group
$D_{2N}$ on $\R^N$. We then use the representation-theoretic analysis
of the second moment developed in \cite{bendory2022sample} to give an
equivalent but simpler expression for the auto-correlation (Section
\ref{sec.secondmoment}). To complete the proof we extend the
algebro-geometric techniques of \cite{conca2015algebraic}.

\section{Background in Algebraic Geometry}
This section provides a brief overview of results from algebraic geometry that play a crucial role in our investigations. Let $\K$ be a field (usually $\R$ or $\C$). We endow $\K^N$ with the
Zariski topology defined by letting the closed sets be given by the
vanishing loci of polynomials in $\K[x_1,\dots, x_N]$. We call such
closed sets \textit{affine varieties}. Moreover, if these polynomials
are homogeneous, then their vanishing locus defines a subset of the
projective space $\Pro(\K^N)$, and is called a \textit{projective
  variety}. Open subsets (under the subspace topology) of affine
(resp. projective) varieties are called quasi-affine (resp. quasi-projective)
varieties.
It is also noteworthy that Zariski closed sets of $\R^N$ or $\C^N$
are also closed in the Euclidean topology, and their complements,
when non-empty, are dense in the Euclidean topology.

We say that a variety is {\em irreducible} if it is irreducible as a
topological space in the Zariski topology, meaning that it cannot be
decomposed into a union of two proper closed subsets. With this
definition, any non-empty Zariski open set in an irreducible variety
is necessarily dense.

We say that a property (*) holds {\em generically} on an irreducible variety $X$
if the set of points $p \in X$ satisfying (*) contains a non-empty Zariski open set. Note that this is stronger than the assertion that the set of points
for which (*) holds is dense in the Zariski topology. For example
the integers are Zariski dense in $\R$ but do not contain a non-empty Zariski open set of $\R$.

In our proofs we rely heavily on
the notion of the dimension of a variety. For an elementary introduction to
the  many
equivalent definitions of dimension, see \cite[Chapter 9]{Harris}.
In particular, we call the vanishing locus of
a single non-constant polynomial a {\em hypersurface} and the dimension of a
hypersurface in $\K^N$ is always $N-1$. We also take advantage of the
interplay between real and complex varieties. Given a complex variety
$X$, we denote by $X_\R$ the real points of $X$.

A {\em morphism} $f \colon X \to Y$ of varieties is a continuous
map with the property that for every point $p \in X$ there is a Zariski open
set $U$ containing $p$ and polynomial functions $g,h$ such that $h$ does not vanish on $U$ and $f(q) = (g/h)(q)$ for all points in $U$.
A morphism $f \colon X \to Y$ is {\em dominant} if $f(X)$ is a Zariski dense
subset of $Y$. In this case there is a non-empty Zariski open set $U \subset Y$
contained in $f(X)$ \cite[Exercise 3.19]{hartshorne}. A surjective morphism is always dominant.

The following property of dimension and morphisms will be used extensively.
\begin{proposition} \label{prop.vakil}
  \cite[Exercise 3.22]{hartshorne}
  If $f \colon X \to Y$ is a dominant morphism of irreducible varieties of dimensions $m$ and $n$
  respectively then
there exists a non-empty Zariski open set $U \subset Y$
    such that for all $y \in U$, $\dim f^{-1}(y) = m-n$.
\end{proposition}
We will also use a property of morphisms called {\em flatness}.
The definition is somewhat technical so for brevity
we refer the reader to an advanced textbook in algebraic geometry such as \cite[Section III.9]{hartshorne}. In this paper we use the following
properties of flatness.
\begin{proposition} \label{prop.flatness}
  \phantom\\
  \begin{enumerate}
  \item[(i)] If $f \colon X \to Y$ is a flat morphism of varieties
    then $f(U)$ is Zariski open
    for any Zariski open set $U \subset X$.
  \item[(ii)] A projection $\pi \colon X \times Y \to Y$ is always flat.
  \item[(iii)] If $f \colon X \to Y$ is a morphism such that for every
    $y \in Y$, $f^{-1}(y)$ is a hypersurface of degree $d$ in $\K^N$ or $\Pro ^N$
    then $f$ is flat, and more generally if every fiber $f^{-1}(y)$ is a product of hypersurfaces then $f$ is flat.
  \item[(iv)] If $Y$ is irreducible and $f \colon X \to Y$ is flat
    then $\dim f^{-1}(y)$ is constant for all $y \in Y$ and every irreducible
    component of $X$ has dimension equal to $\dim Y + \dim f^{-1}(y)$.
  \end{enumerate}
\end{proposition}

\section{Reduction to the case of overlapping frames}
\subsection{Ordered frames}
A $N$-element ordered frame in $\K^M$ is a collection over $N$-vectors $\UU = (\alpha_1, \ldots , \alpha_N)$
that span $\K^M$. Given a frame $\UU$ the {\em analysis operator} of $\UU$ is
the linear map $L_\UU \colon \K^M \to \K^N$ defined by $y \mapsto (\langle x, \alpha_1 \rangle, \ldots , \langle x, \alpha_N\rangle)$. To give an $N$-element frame $\UU$ in $\K^M$ is equivalent to giving a
a full rank $M \times N$ matrix $A_\UU = \begin{pmatrix} \alpha_{1,1} & \ldots & \alpha_{1,N}\\ \alpha_{2,1} & \ldots & \alpha_{2,N}\\
\ldots & \ldots & \ldots\\
\alpha_{M,1} & \ldots & \alpha_{M,N}
\end{pmatrix}$
since the column vectors of this matrix are vectors $\UU =(\alpha_1, \ldots , \alpha_N)$ that span 
$\K^M$. Note
that the rows of the matrix $A_\UU$ are an ordered basis for the range of the analysis operator
$L_\UU$ which is an $M$-dimensional linear subspace of $\K^N$. We sometimes denote this subspace
as $\L_{\UU}$.
\subsection{Overlapping frames}
A pair of $N$-element frames 
$\UU=(\alpha_1, \ldots , \alpha_N), \VV= (\beta_1, \ldots , \beta_N)$ is
$s$-overlapping if the ranges of their analysis operators intersect in linear subspace of dimension $s$ 
and there are $s$-element sets of indices $T_1= (i_1, \ldots , i_s)$ and $T_2 = (j_1, \ldots , j_s)$ such
that for each $\ell$, the $i_m$-th entry of $\alpha_\ell$ and the $j_m$-th entry of $\beta_\ell$ agree.
For each pair of $s$-element subsets $T_1, T_2 \subset [1,M]$ the set of $s$-overlapping frames whose vectors agree in the index sets $T_1, T_2$
is parametrized as Zariski open set of $\K^{(2M-s) \times N}$. In particular, the set of $s$-overlapping frames is a union of irreducible varieties each isomorphic to an open set $\K^{(2M-s) \times N}$.

\subsection{Sparsity and overlapping frames}
The set of ordered bases of $\K^N$ is equivalent to the variety $\GL_N(\K)$ of invertible $N \times N$
matrices, since the rows of any invertible matrix form a basis for $\K^N$. To prove Theorem~\ref{thm.realsparse} and 
Theorem~\ref{thm.complexsparse} we will to prove that there is a Zariski open set $U \subset \GL_N(\K)$ such that
for $M$ as given in the statement of the theorems, and every pair of $M$-element subsets $S_1, S_2 \subset [1,N]$ and vectors $x \in \L_{S_1}(\V),y \in \L_{S_1}(\V)$
if $P_x = P_y$ then $x \in t y$ with $t \in S^1 \cap \K$. (Note this necessarily implies that $x, y \in \L_{S_1}(\V) \cap
\L_{S_2}(\V)$.)
Given an ordered basis of $\V=\{v_1, \ldots , v_N\}$ of $\K^N$ and an
an $M$-element subset $S \subset [1,N]$ let $A_S(\V)$ be the $M \times N$
matrix whose rows consist of the vectors $v_{i_1}, \ldots v_{i_M}$ where $i_1 < i_2 < \ldots < i_M$ are the
elements of $S$. The columns of $A_S(\V)$ form a set of $N$ vectors in $\K^M$ which, because $A_S(\V)$ has full rank, are
a frame which we denote $\UU_S$. If the entries of the vector $v_i$ are denoted by $v_i = (v_{i,1}, \ldots , v_{i,N})$
then the $\ell$-th column of $A_S(\V)$ is the vector $\alpha_\ell = (v_{i_1,\ell}, \ldots v_{i_M,\ell})$. 

\begin{example} \label{ex.overlap}
Here is an example to illustrate how we obtain overlapping frames from the choice of two $M$ element subsets
of an ordered basis for $\K^N$. Let $\V$ be the ordered basis $\V$ for $\R^5$  consisting of the row vectors of the invertible matrix
$$A =\left(
\begin{array}{ccccc}
 4 & -7 & 5 & 8 & -3 \\
 -4 & -10 & 8 & 0 & 8 \\
 -1 & 3 & 9 & 10 & 0 \\
 2 & -6 & 1 & 0 & -1 \\
 -2 & -8 & 1 & 3 & 6 \\
\end{array}
\right)$$
Consider the two element subsets $S_1 = \{1,4\}$, $S_2 = \{4,5\}$  which intersect in the 1-element set $\{4\}$. 
Then we have $A_{S_1}(\V)= \begin{pmatrix} 4 & -7 & 5 & 8 & -3\\ 2 & -6 & 1 & 0 & -1 \end{pmatrix}$ which is the matrix whose
columns are the first and fourth entries of the columns of $A$. Likewise,
$A_{S_2}(\V) = \begin{pmatrix}  2 & -6 & 1 & 0 & -1 \\ -2 & -8 & 1 & 3 & 6 \end{pmatrix}$.
The corresponding frames associated to these subsets are $\UU = \{ (4,2), (-7,-6),(5,1), (8,0),(-3,1)\}$ and $\VV = \{(2,-2), (-6,-8), (1,1),(0,3), (-1,6)\}$.
Then $\UU, \VV$ is a pair of 1-overlapping frames with
$T_1 = \{2\}$ and $T_2= \{1\}$ since the second entry of every vector in $\UU$ equals to the first entry of the corresponding vector in 
$\VV$. The range of the analysis operator $L_{\UU}$ is the two dimensional subspace of $\R^5$,  $\L_{\{1,4\}}(\V) = \langle (4,-7,5,8,3), (2,-6,1,0,-1)\rangle$ and the range of the analysis operator $L_{\VV}$
is $\L_{\{4,5\}}(\V) = \langle (2,-6,1,0,-1), (-2,-8,1,3,6) \rangle$. These subspaces intersect in the line in $\R^5$  spanned by the vector $(2,-6,1,0,-1)$.
\end{example}

\subsection{A statement for overlapping frames}
Now suppose $S_1, S_2$ are two $M$-element subsets of $[1,N]$ with $S_1 \cap S_2 = s$.
Let $\UU$ be the frame associated to the matrix $A_{S_1}(\V)$ and $\VV$ the frame associated to the matrix $A_{S_1}(\V)$.
The frames $\UU$, $\VV$ are $s$-overlapping frames as illustrated in Example~\ref{ex.overlap}.
To prove Theorem~\ref{thm.realsparse} and Theorem~\ref{thm.complexsparse} we will prove the following statement about $s$-overlapping frames for $0 \leq s \leq M$.\\

For $M$ in the suitable range and for every $0 \leq s \leq M$ there is a Zariski dense open set in the variety of $s$-overlapping pairs of frames such that if $(\UU, \VV)$ is a pair of overlapping
frames in this open set then:\\

For all $x,y \in \K^M$ (resp. for generic $(x,y) \in \K^M$) $L_{\UU}(x)$ and $L_{\VV}(y)$ have the same power spectrum if and only
if $L_{\UU}(x) = t L_{\VV}(y)$ for some $t \in S^1 \cap \K$. (Here $L_{\UU}$ and $L_{\VV}$ are the analysis operators of the frames $\UU$ and $\VV$ respectively.)\\

\section{Fourier phase retrieval for complex overlapping frames - Proof of Theorem \ref{thm.complexsparse}} \label{complexcase}
We prove Theorem \ref{thm.complexsparse} first
since we can use 
techniques already established in the literature for the frame phase retrieval
problem \cite{conca2015algebraic}.

If $\UU = (\alpha_1, \ldots , \alpha_N)$ is an $N$-element frame on
$\C^M$ represented by an $M \times N$ matrix $A_{\UU}$
let $\UU' = (\alpha'_1,\ldots , \alpha'_N)$
be the frame represented by the matrix $A_{\UU} F$ where $F$ is the $N \times N$ discrete Fourier
transform matrix.
If $x \in \C^M$ then
$|F L_{\UU}(x)|^2 = |L_{\UU'}(x)|^2$ where as usual the absolute value is taken componentwise.
Moreover if $\UU$ and $\VV$ are $s$-overlapping frames then the corresponding frames $\UU'$ and $\VV'$ obtained by applying the discrete Fourier transform
are also $s$-overlapping.
Therefore, Theorem \ref{thm.complexsparse}(i) follows from the following
statement.
\begin{enumerate}
\item[(F)] \label{item.two}
  If $N > 4M-3$ and $\UU = (\alpha_1, \ldots , \alpha_N)$
  and $\VV = (\beta_1, \ldots , \beta_N)$ are generic $s$-overlapping frames with $0 \leq s \leq M$ then 
$|L_{\UU}(v)| \neq |L_{\VV}(w)|$ unless $L_{\UU}(v) = t L_{\VV}(w)$ for some $t \in S^1$.
\end{enumerate}

Likewise, Theorem \ref{thm.complexsparse}(ii) follows from the following statement.
\begin{enumerate}
\item[(Fg)] If $N > 2M-1$ and $\UU,\VV$ are generic $s$-overlapping frames with $0 \leq s \leq M$
then for a generic vector $x \in \C^M$,
  $|L_{\UU}(x)| = |L_{\VV}(y)|$ if and only
  if $L_{\UU}(x) = t L_{\VV}(y)$ for some $t \in S^1$.
\end{enumerate}

\begin{remark}
As noted earlier, (F) can be deduced from
\cite[Theorem 2.2]{wang2014phase}. Also, when $s=M$ and then $\UU = \VV$ and (F) holds with the bound $N > 4M-5$ by~\cite[Theorem 1.1]{conca2015algebraic}. Likewise when $\UU = \VV$ (Fg) appears as
\cite[Theorem 3.3]{balan2006signal}.  
\end{remark}
\subsection{Phase retrieval for pairs of frames}
Let $\UU = (\alpha_1,\ldots \alpha_N)$ and $\VV =(\beta_1, \ldots ,
\beta_N)$ be two $s$-overlapping $N$-element frames on $\C^M$.
After simultaneously reordering the entries of the vectors in $\UU$ and $\VV$ we may assume that that the first $s$ entries of the $\ell$-th vector
in $\UU$ equal to the first $s$ entries of the $\ell$-th vector in $\VV$. In other words, 
$\alpha_{ij} = \beta_{ij}$ for $ 1 \leq i \leq s$ and $1 \leq j \leq N$.
(When $s=0$, we understand $\UU$ and $\VV$ to be arbitrary frames.) We will denote the variety parametrizing such frames as $F(M,N,s)$. It is a Zariski open set of $\K^{(2M-s) \times N}$ parametrizing $(2M-s)\times N$ matrices $A$ such that the columns of the $M \times N$ submatrices $A_1$ and $A_2$ are $N$-element frames in $\K^N$, where $A_1$ consists of the first $M$-rows and $A_2$ consists of the first $s$ and last $M-s$ rows
of $A$.

Consider the map 
$$ \mathcal{A}_{\UU\VV}: (\C^M/S^1) \times (\C^M/S^1) \to \R^N,
(x,y) \mapsto |L_{\UU}(x)|^2 - |L_{\VV}(y)|^2$$
where $L_\UU$ and $L_\VV$ are the analysis operators of the frames $\UU$ and $\VV$ respectively.
 To prove statement (F) we need to show that if $N > 4M-3$ then
    for a generic pair of $s$-overlapping frames $\UU, \VV$,\\
\centerline{(*) $\mathcal{A}_{\UU \VV}(x,y) = 0$ if and only if $x =y$.}

If
we identify $\C^M/S^1$ with the space $\H_M^1$ of rank-one
Hermitian matrices then $\mathcal{A}_{\UU\VV}$ extends (cf. \cite{bandeira2014saving})
to a linear map
$$\mathcal{A}_{\UU\VV}\colon \mathbb{H}_M \times \mathbb{H}_M \to \R^N, (A ,B) = (\alpha_i^* A \alpha_i - \beta_i^* B \beta_i)_{i=1}^{N}$$
where $\H_M$ is the real vector space of $M \times M$ Hermitian matrices. 
Explicitly the map sends a pair $(A,B)$ of $M \times M$ Hermitian matrices to the vector in $\R^N$ whose $i$-th component
    is the scalar $(\alpha_i^* A \alpha_i - \beta_i^* B \beta_i)$. (In this notation we view $\alpha_i, \beta_i$ as column vectors and we note that if $x$ is a column vector and $A = xx^*$ then $\alpha_i^* A \alpha_i = |\langle x, \alpha_i \rangle|^2 $.)

Let ${\mathbb U} \subset \H^1_M \times \H^1_M$ be the open
set of pairs of distinct rank-one Hermitian matrices. The following is immediate from our discussion.
\begin{proposition}\label{cxinjcondition}
  Condition (*) holds if and only if ${\mathbb U}$ does not intersect
  the kernel of the linear map $\mathcal{A}_{\UU \VV} \colon \H_M \times \H_M \to \R^N$.
\end{proposition}  
Thus, to prove statement (F) or, equivalently, Theorem \ref{thm.complexsparse}(i)
we must prove the following result.
\begin{theorem}\label{genericpairsab}
If $N > 4M-3$ and $\UU$ and $\VV$ are a generic pair of $s$-overlapping frames then the set ${\mathbb U}$ does not intersect the kernel
  of the linear map $\mathcal{A}_{\UU\VV} \colon \H_M \times \H_M \to \R^N$
\end{theorem}
Likewise, statement (Fg), or equivalently, Theorem \ref{thm.complexsparse}(ii) follows
from the following result.
\begin{theorem}\label{generica}
  Suppose $N > 2M-1$ and $\UU,\VV$ are generic pair of $s$-overlapping frames.
If $A \in \H_M^1$ is generic there is no 
$ B \neq  A  \in \mathbb{H}_M^1$ such that $\alpha_i^*A \alpha_i =
\beta_i^*B \beta_i$ for $i = 1, \ldots N$.
\end{theorem}

\begin{remark} If the overlap is $s < M$ then for generic $A$,
  $(\alpha_i^* A \alpha_i)_{i=1}^N \neq
  (\beta_i^* A \beta_i)_{i=1}^N$ so the assumption that
  $ B  \neq  A $ is superfluous.  Translated
  back into our original setup, Theorem \ref{generica} tells us that
  if $N> 2M-1$ and $S \subset [1,N]$ with $|S| = M$ and $x \in
  \L_S(\V)$ is generic, then any vector in $\L_S(\V)$ with the same power
  spectrum must be obtained from $x$ by multiplication by a global
  phase. Likewise, given two distinct $M$-element subsets $S_1, S_2 \subset [1,N]$
  and a generic signal $x \in
  \mathbb{L}_{S_1}(\V)$, there is no $y \in \mathbb{L}_{S_2}(\V)\setminus \L_{S_1}(\V)$
  with the same power spectrum as $x$.
\end{remark}

\subsection{Proof of Theorems \ref{genericpairsab}, \ref{generica}} 
  
Recall that $F(M,N,s)$ denotes the set of $N$-element $s$-overlapping frames
of vectors in $\C^M$ whose first $s$ coordinates match.
If $(\UU, \VV) \in F(M,N,s)$ we 
can decompose the corresponding $M \times N$ matrices $A_\UU$ and $A_\VV$ in terms of their real and imaginary parts
as  $A_\UU = U+iU'$ and $A_\VV = V+iV'$ .
Since the first $s$-rows of $A_\UU$ and $A_\VV$ are equal,
the set of such frames depends on $4NM -2sN$ real parameters.
Let $\Lambda_s \subset (\C^{M \times N})^4$ be the parameter space
of four-tuples of complex $M\times N$ matrices $(U,U', V, V')$ such that
the first $s$ rows of $U$ and $V$ are equal and the first $s$ rows of $U'$ and $V'$ are equal.
By construction $\Lambda_s$ is a complex linear subspace of $(\C^{M \times N})^4$ of dimension
$4MN - 2sN$ and the set of $s$-overlapping frames is the real Zariski open
set of $\Lambda_s(\R)$ corresponding to tuples $(U,U', V, V')$ such that
$\UU = U + iU', \VV = V + iV'$ have full rank and the intersection of the
ranges of $\UU$ and $\VV$ has dimension exactly $s$.

Following \cite{conca2015algebraic} we define the following incidence relation.
\begin{definition}\label{defbmn}
  Let $\b_{M,N,s}$ denote the subset of
  $\Pro(\Lambda_s) \times \Pro(\hermdecomp{C} \times \hermdecomp{C})$ consisting of tuples of matrices $([U,U', V,V'], [A,A',B,B'])$ for which
	\begin{equation*}
	\begin{split}
	& (1) A + i A' \neq B + iB'\\
	&(2) \ \rank(A+iA')=\rank(B+iB') = 1\\
	&(3) \  u_n^T Au_n + {u'}_n^TA{u'}_n - 2u_n^TA'{u'}_n = v_n^T B v_n + {v'}_n^TB{v'}_n - 2v_n^TB'{v'}_n \qquad \text{ for all } 1\leq n \leq N
	\end{split}    
	\end{equation*}
\end{definition}
We immediately see that $\b_{M,N,s}$ is a quasi-projective subvariety
of $$\Pro(\Lambda_s) \times \Pro(\hermdecomp{C} \times \hermdecomp{C})$$
since constraints $(2)$ and $(3)$ are defined by the vanishing of
homogeneous polynomials; whereas constraint $(1)$ is an open
condition. Let $\pi_1$ be the projection onto $\Pro(\Lambda_s)$ and
$\pi_2$ the projection to $\Pro(\hermdecomp{C} \times \hermdecomp{C})$
\begin{proposition} \label{injectiveiff}
  Given a pair of $s$-overlapping frames $\UU = U+iU'$, $\VV = V + iV'$, 
  $\ker \mathcal{A}_{\UU\VV}$  does not contain any pairs of distinct rank-one Hermitian matrices
  if and only if $(U,U',V,V')$ is not in the projection $\pi_1((\b_{M,N,s})_\R)$.
\end{proposition}
\begin{proof}
Consider the incidence correspondence
$$I = \{(\UU,\VV, \mathbf{A} ,\mathbf{B}) \ | \ \mathbf{A} \neq \mathbf{B} , \ \rank( \mathbf{A} ) =
\rank(\mathbf{B} ) = 1, \ \alpha_n^* \mathbf{A} \alpha_n = \beta_n^* \mathbf{B} \beta_n\} \subset
F(M,N,s) \times \mathbb{U}.$$
The image of $I$ under the projection $I \to
F(M,N,s)$ is the set of $s$-overlapping frames for which there exists
a pair distinct rank one Hermitian matrices $(\mathbf{A},\mathbf{B})$
lying in $\ker \mathcal{A}_{\UU\VV}$.
If we write
$\UU = U+iU'$, $\VV=V+iV'$, $ \mathbf{A} = A+iA'$ and $ \mathbf{B} = B+iB'$ with $A,B$
symmetric and $A',B'$ skew-symmetric then $I$ is linearly isomorphic
over $\R$ to the the real incidence correspondence
	$$\mathcal{I} = \{(U,U',V,V', A,A',B,B')| (1)-(3) \text{ hold}\}$$
seen as a subset of the real vector space
$\Lambda_s(\R) \times  \hermdecomp{R}\times \hermdecomp{R}$.
Since $\Pro(\mathcal{I}) = \b_{M,N,s}(\R)$ 
the proposition follows.
\end{proof}

Now we compute the dimension of $\b_{M,N,s}$, which will bound the dimension of $\pi_1(\b_{M,N,s})$. To this end, we begin with a lemma.
\begin{lemma}\label{dim}
  Let $ A ,  B  \in \C^{M\times M}$ be
  distinct matrices.
  The algebraic subset of $\C^{4M}$ defined by the ideal 
in $\C[u_1, \ldots , u_M, u'_1, \ldots , u'_M, v_1, \ldots , v_M, v'_1, \ldots , v'_M]$ generated by the following linear and quadratic polynomials.
  \begin{enumerate}
  \item[(1)] $u_n - v_n$, $u'_n - v'_n$ for $n =1, \ldots , s \leq M$.
  \item[(2)] $q(u,u',v,v') = (u-iu')^T  A  (u+iu') - (v-iv')^T B (v+iv')$,
    where $u=(u_1, \ldots , u_M)$, etc.
  \end{enumerate}
has dimension exactly $4M-2s-1$.
\end{lemma}
\begin{proof}
  The $2s$ linear forms (1) are clearly independent and the linear subspace
  defined by these forms has dimension $4M -2s$. To complete
  the proof we need to show that (2) imposes one additional
  condition. Since the effect of the equations (1) is to replace
  the variables $v_n$ (resp. $v'_n$) with $u_n$ (resp. $u'_n$) for $n =1,\ldots
  s$, this is equivalent to showing that for any $s \leq M$,
  the quadratic form $q(u,u',v,v')$ is non-vanishing when
  $v=(u_1, \ldots , u_s, v_{s+1}, \ldots , v_M)$ and $v' = (u'_1,\ldots , u'_s,
  v'_{s+1}, \ldots , v'_M)$. Clearly it is sufficient to show this
  in the case where $s=M$. In other words it suffices to show that
  the quadratic form
  $q(u,u') = (u + i u')( A - B )(u + iu')$ is non-vanishing
  if $ A  -  B  \neq 0$. This follows from \cite[Lemma 3.5]{conca2015algebraic}.
\end{proof}

\begin{proposition}\label{dimbmn}
	The dimension of $\b_{M,N,s} = 4M+4MN-N-2sN-4$.
\end{proposition}
\begin{proof}
  Let $\b'_{M,N,s}$ be the quasi-projective subvariety of $\Pro(\Lambda_s) \times
  \Pro(\C^{M\times M} \times \C^{M\times M})$ consisting of tuples of matrices $([U,U',V,V'],[ A ,  B ])$ satisfying 
	\begin{equation*}
	\begin{split}
	&(1) \  A  \neq  B \\
	&(2) \ \rank( A ) = \rank
	( B )=1\\
	&(3) \ (u_n-iu'_n)^T A (u_n+iu'_n) = (v_n-iv'_n)^T B (v_n+iv'_n) \qquad \text{ for all } 0\leq n \leq N-1
	\end{split}    
	\end{equation*}

The varieties $\b_{M,N,s}$ and $\b'_{M,N,s}$ are linearly isomorphic
since we can identify $(\C_{\mathrm{sym}}^{M\times N} \times
\C_{\mathrm{skew}}^{M\times N})^2$ with $(\C^{M\times M})^2$ as noted
in the proof of \cite[Theorem 3.4]{conca2015algebraic}. Thus it
suffices to prove that $\b'_{M,N}$ has the desired dimension. We now
view $\B'_{M,N,s}$ as a subvariety of $\Pro(\Lambda_s) \times \Pro(U^1_M)$
where $U^1_M$ is the quasi-affine variety of pairs of $(A,B)$ of distinct
complex rank-one matrices.

Let $\pi_1$ be the projection onto the
first coordinate and $\pi_2$ be the projection onto the second
coordinate. We will compute the dimension of $\b'_{M,N,s}$ using the
dimension of its projection $\pi_2(\b'_{M,N})$ and the dimension of
its fibers $\pi_2^{-1}([A,B])$ for $[A, B] \in \Pro(U^1_M)$.

\begin{lemma} Any pair of rank $1$
matrices $[(A, B)]$ with $ A \neq B $ belongs to the image of
$\pi_2$; i.e. $\pi_2(\b'_{M,N,s}) = \Pro(U^1_M)$.
\end{lemma}

\begin{proof} If $s=M$ take any non-zero vectors $(u,u')$ satisfying
the quadratic equation $(u + iu')( A - B )(u + iu')^T$ and take $U =
V= (u,\ldots, u)$ and $U' = V' =(u', \ldots , u')$.  If $s< M$ take
any vectors $u,u', v,v' \in \C^M$ such that the first $s$ coordinates
of are zero and which also satsify the quadratic equation  $$(u
+ i u')A(u + i u')^T - (v + iv')B(v + i v')=0.$$ and let $U = (u
\ldots u), U'=(u' \ldots u'), V = (v \ldots v),V' = (v' \ldots v')$
\end{proof}

The set of pairs of distinct rank-one matrices $(A,B)$ 
dimension $2(2M-1)=4M-2$ by
\cite[Prop 12.2]{Harris} and thus we conclude that
$$\dim(\pi_2(\b'_{M,N})) = \dim \Pro(U^1_M) = 4M-3.$$

Now we focus on the fibers of
$\pi_2$.  Fix $[(A, B)] \in \pi_2(\b'_{M,N})$. We know
by Lemma \ref{dim}
that the fiber over a pair $[(A,B)]$ is a product of $N$ degree-two hypersurfaces
in $\Pro(\C^{4M-2s})$. Therefore, 
every fiber of $\pi_2 \colon \b'_{M,N,s} \to U^1_M$ has dimension
$N(4M-2s-1)-1$ and moreover the map $\pi_2 \colon \b'_{M,N,s} \to U^1_M$ is
flat by Proposition \ref{prop.flatness}(iii). Hence 
by Proposition \ref{prop.flatness}(iv) 
we conclude that every irreducible component of $\b'_{M,N,s}$
has dimension $$\dim \Pro(U^1_M) + N(4M -2s -1) -1 = 4MN +4M -N -2sN -4.$$
Therefore, 
$\dim(\b'_{M,N}) = 4M+4MN-N-2sN-4$
since the dimension of a variety is equal to the maximum of the dimensions of its irreducible components. 
\end{proof}

\begin{proof}[Proof of Theorem \ref{genericpairsab}]
By Proposition \ref{injectiveiff} a four-tuple of real $M\times N$
matrices $(U,U',V,V')$ with $\UU=U+iU'$ and $\VV = V+iV'$ for which
$\mathcal{A}_{\UU\VV}$ is not injective gives a point
$[U,U',V,V'] \in \pi_1((\b_{M,N,s})_\R)\subset
(\pi_1(\b_{M,N,s}))_{\R}$. Since the dimension of this projection is
bounded by the dimension of the original variety $\b_{M,N,s}$
\cite[Cor 11.13]{Harris}, we see that
	$$\dim(\pi_1(\b_{M,N,s})) \leq \dim(\b_{M,N,s}) = 4M+4MN-N-2sN-4$$ 
In particular,  when $N>4M-3$, the dimension of this projection is
{\em strictly less} than $\dim \Pro(\Lambda_s)= 4MN -2sN -1$.
Hence if $N > 4M-3$ the projection is contained in a proper algebraic subset of $\Pro(\Lambda_s)$, which means that for a generic pair $(\UU,\VV)$ of $s$-overlapping
frames there is no pair $(A,B)$ of distinct rank-one Hermitian matrices
in $\ker \A_{\UU\VV}$.
\end{proof}

\begin{proof}[Proof of Theorem \ref{generica}]
Consider the morphism
$\pi_1 \colon \b'_{M,N,s} \to \Pro(\Lambda_s)$.
If the image of $\pi_1$ lies in a proper algebraic subset
of $\Pro(\Lambda_s)$ then as in the proof of Theorem \ref{genericpairsab}
we can conclude that for a generic pair $(\UU, \VV)$ of $s$-overlapping frames there are no pairs of distinct rank-one Hermitian matrices $(A,B)$
such that $(\alpha_i^*A \alpha_i)_{i=1}^N = (\beta_i^* B \beta_i)_{i=1}^N$.
Hence, we can assume that the image of $\pi_1$ is not contained in a proper
algebraic subset of $\Pro(\Lambda_s)$; i.e. the map $\pi_1$ is dominant.

Applying Proposition \ref{prop.vakil} to each irreducible component of
$\b'_{M,N,s}$ that dominates $\Pro(\Lambda_s)$
we conclude that there is a Zariski open set of points $q = [U,U',V,V'] \in \Pro(\Lambda_s)$ such that the dimension of the fiber $\pi_1^{-1}(q)$ equals
$$\dim \b'_{M,N} - \dim \Pro(\Lambda_s) =
(4M + 4MN -N - 2sN - 4) - (4MN - 2sN {-1}) =4M-N-3$$
with the understanding that if $4M-N{-3} < 0$, 
the fiber is empty.

If $q = [U,U',V,V'] \in \Pro(\Lambda_s)$ then its fiber $\pi_1^{-1}(q)$ 
is the quasi-projective variety
$$\Gamma = \{[ A , B ] \in \Pro(\M_M^1 \times \M_M^1) |  A \neq  B ,
        \ (U-iU')^T A (U+iU') = (V-iV')^T B (V+iV)\}$$
where $\M_M^1$ denotes the quasi-affine variety of rank-one
        $M \times M$ matrices. If $q$ is chosen generically, then the previous
	paragraph implies that  $\dim \Gamma =4M-N{-3}$.

Let $p_1: \Gamma \to \Pro(\M_M^1)$ be the
projection to the first coordinate. (Note that this projection is well
defined because $A$ is never 0 since it is a rank-one matrix.)  Since
$\dim \Gamma = 4M-N{-3}$ the closure $\overline{p_1(\Gamma)}$ will be a
proper algebraic subset of $\Pro(\M_M^1)$ if
$$\dim \Pro(\M_M^1) = 2M-2 > 4M -N{-3};$$ i.e. if $N > 2M{-1}$.

This implies that if $N > 2M{-1}$, for
a generic choice of $[U,U',V,V']$, the set of rank-one matrices$A$
such that there is no rank-one matrix $B$ satisfying
$(U-iU')^T A (U+iU') = (V-iV')^T B (V+iV')$ is a
Zariski open subset of $\Pro(\M_M^1)$. Restricting 
$[U,U',V,V']$ to the real points of $\Lambda_s$ corresponding to $s$-overlapping
frames and restricting $A$ to the real points $\M^1_M$ corresponding
to rank-one Hermitian matrices
we see that for a 
a generic pair of $N > 2M{-1}$ element frames $\UU, \VV$ and a
generic choice of $A \in \H_M^1$, there is no
$B\neq A \in \H_M^1$ such that
$(\alpha_i^*A \alpha_i)_{i=1}^N = (\beta_i^* B \beta_i)_{i=1}^N$
\end{proof}

\section{Fourier phase retrieval for overlapping real frames - proof of Theorem \ref{thm.realsparse}}\label{realcase}
\subsection{Setup and notation}
We now consider the case where $\K=\R$ where the situation is
more subtle and requires new ideas.
The difficulty is that if $\UU$ is a real frame with corresponding matrix
$A_{\UU}$ then $A_{\UU}F$ is not the matrix of a real frame, and it is also not the matrix of a `generic' complex frame. To circumvent this difficulty
we start by replacing the power spectrum with the periodic auto-correlation.
Recall from \cite{bendory2020toward}
that if $x=(x[0], \ldots , x[N-1]) \in \C^N$ then the periodic auto-correlation
$a_x \in \C^N$ is the vector whose $\ell$-th component is
$$a_x[\ell] = \sum_{n=0}^{N-1} x[n]\overline{x[\ell + n]}.$$
where all indices are taken $\bmod \, N$.
Since $Fa_{x} =
|Fx|^2$ for any $x \in \C^N$, the periodic auto-correlation gives the same information as the power spectrum.

If we restrict to $x \in \R^N$ then $a_x$ is the collection
of real quadratic functions
$$a_x[\ell] = \sum_{n=0}^{N-1} x[n]x[\ell + n].$$
For any $x\in \R^N$ the auto-correlation can
be thought of as a quadratic map $a_x:\R^N \to \R^{\lfloor N/2 \rfloor
  + 1}$ since $a_x[\ell] = a_x[N-\ell]$,
and it suffices to prove the corresponding restatement of Theorem~\ref{thm.realsparse} with the auto-correlation replacing the power spectrum.

Unfortunately, the auto-correlation is a collection of
$\lfloor N/2 \rfloor + 1$ quadratic forms in $N$-variables. Because each variable appears in each of the forms
it is difficult to directly compute the dimension of the corresponding
incidence varieties as we did in Section \ref{complexcase}.
In the following subsection, we leverage
techniques and ideas from \cite{bendory2022sample} which will allow
us to replace the auto-correlation with a simpler set of real quadratic functions.

\begin{remark} For the sake of notational simplicity, we will assume in our further discussion that
  $N$ is even so that the auto-correlation is given by $N/2 +1$ quadratic functions.
  If $N$ is odd, the auto-correlation is given by $(N+1)/2$ quadratic functions
  which affects the bounds in Theorem \ref{real_genericL} with $4M-5$ replacing
  $4M-6$, Theorem \ref{real_genericpairs} with $4M-3$ replacing $4M-4$,
  and Theorem \ref{real_generica} with $2M-1$ replacing $2M-2$.
\end{remark}

\subsection{Second moment for dihedral group actions} \label{sec.secondmoment}
This subsection focuses on reinterpreting the auto-correlation
function as the second moment under the action of the dihedral group
on $\C^N$ or $\R^N$. By doing
so, we can use the results of \cite{bendory2022sample}
on the structure of the second moment to obtain 
a simpler system of homogeneous quadratic functions
which gives the same information as the periodic auto-correlation.

Following standard terminology, we refer to a vector space $V$
equipped with a linear action of a group $G$ as a
{\em representation} of $G$. Although we are inherently interested in
real signals, we illustrate the reduction for complex signals, where
the theory is cleaner and it is easier to understand the core
ideas at play. At the end of this section we specialize
to the 
case of real signals to get our desired simplification.

Let $D_{2N}$ denote the dihedral group of order $2N$, with presentation
$$D_{2N} = \{r, s | r^N=s^2 = 1, rs=sr^{-1}\}$$
and let $D_{2N}$ act on $\C^N$ via
\begin{equation*}
\begin{split}
(r\cdot x)[\ell] &= x[\ell+1]\\
(s\cdot x)[\ell] &= x[N - \ell]
\end{split}
\end{equation*}
where again all indices are taken $\bmod \, N$.

This makes $\C^N$ an $N-$dimensional representation of $D_{2N}$. Define the second moment of a vector $x$ as the function $m^2_x: \C^N \to \C^{N\times N}$ given by
$$m^2_x = {1\over{|D_{2N}|}}\sum_{g\in D_{2N}}(g \cdot x)(g \cdot x)^*.$$

For an arbitrary complex vector the second moment captures less information than the periodic auto-correlation. However the following
proposition shows that if $x$ is real then the second moment is equivalent to the periodic auto-correlation \cite[Section 2.4.2]{bendory2022sample}.
\begin{proposition}\label{autocormoment}
	Given $x \in \R^N$, the second moment $m^2_x$ captures the same information as the periodic auto-correlation function $a_x$. 
\end{proposition}
\begin{proof}
Let $H= \Z_N$ denote the cyclic subgroup of $D_{2N}$ generated by $r$. We decompose the second moment according to the cosets of $H$ in $D_{2N}$
$$m^2_x = {1\over{2N}}\left(\sum_{g\in H}(g \cdot x)(g \cdot x)^* + \sum_{g \in Hs} (g \cdot x)(g \cdot x)^*\right) = {1\over{2N}}(S_1 + S_2)$$
where $S_1, S_2$ are $N \times N$ Hermitian matrices.

For any $(\alpha, \beta) \in [0,N-1] \times [0, N-1]$
we can write
$$S_1[\alpha, \beta] = \sum_{p=0}^{N-1} x[\alpha+p] \overline{x[\beta + p]} = \sum_{n=0}^{N-1}  x[n]\overline{x[n + (\beta - \alpha)]} = a_{x}[\beta -\alpha],$$
where we reindexed the first sum with $n=\alpha+p$.
Likewise, 
$$S_2[\alpha, \beta] = \sum_{p=0}^{N-1} {x[N-(\alpha +p)]}\overline{x[N-(\beta+p})] = \sum_{n=0}^{N-1}x[n]\overline{x[(n+(\alpha - \beta)]} = a_{x}[\alpha-\beta],$$
where we reindexed the first sum with $n = N-(\alpha +p)$.
Putting it all together, we see that
$$m^2_x[\alpha,\beta] = {1\over{N}}\left(a_x[\beta-\alpha] +
a_x[\alpha -\beta]\right) .$$
When $x \in \R^N$, $a_x[n] = a_x[-n]$ where the indices are taken $\bmod\,N$,
so $$m_x^2[\alpha,\beta] = {2\over{N}}a_x[\beta -\alpha].$$
\end{proof}

We now consider the action of the dihedral group $D_{2N}$
on $\C^N$ with basis $(v_0, \ldots , v_{N-1})$
where $v_i =Fe_i$ and $F$ is the $N \times N$ discrete Fourier transform matrix. If we expand $z \in \C^N$ in this basis as
$z = \sum_{\ell = 0}^{N-1} z[\ell]v_\ell$
then 
\begin{equation*}
\begin{split}
(r\cdot z)[\ell] &= e^{{2\pi \iota \over{N}}} z[\ell]\\
(s\cdot z)[\ell] &= z[N-\ell].
\end{split}
\end{equation*}
As in \cite[Section 2.4.2]{bendory2022sample}
the $D_{2N}$ action decomposes $\C^{N}$ into a direct sum of irreducible representations
$$\C^N = V_0 \oplus V_1 \oplus \dots \oplus V_{N/2 -1} \oplus V_{N/2}$$
where $V_0 = \mathrm{span}(v_0)$, $V_{N/2} = \mathrm{span}(v_{N/2})$ and $V_k = \mathrm{span}(v_k, v_{N-k})$ for $0<k<N/2$.

With this we are finally ready to perform the key reduction. 
\begin{lemma}\label{bz}
	There exists a quadratic function $b_z:\C^N\to \R_{\geq 0}^{N/2 + 1}$ such that 
	\begin{enumerate}
		\item Each component of $b_z$ is a positive definite quadratic form involving at most two variables
		\item $a_x = a_y \iff b_{Fx} = b_{Fy}$
	\end{enumerate}
\end{lemma}
\begin{proof}
Let $z\in \C^M$ and write 

\begin{equation*}
\begin{split}
z &= z[0]v_0 + (z[1]v_1 + z[N-1]v_{N-1}) + \dots + \left(z\left[\frac{N}{2}-1\right]v_{\frac{N}{2}-1} + z\left[\frac{N}{2}+1\right]v_{\frac{N}{2}+1}\right) + z\left[\frac{N}{2}\right]v_{\frac{N}{2}}\\
&:=f[0] + f[1] + \dots + f\left[\frac{N}{2}-1\right]+ f\left[\frac{N}{2}\right].
\end{split}
\end{equation*}
We see from \cite[Prop 2.1]{bendory2022sample}
that the second moment of the one dimensional irreducible representations $V_0$ and $V_{N/2}$ are determined by $|f[0]|^2 (v_0\otimes \overline{v_0})$ and $|f[N/2]|^2(v_{N/2}\otimes \overline{v_{N/2}})$, while that of the other two dimensional irreducible representations is given by
$$\ip{f[n]}{f[n]}\left(v_n\otimes \o{v_n} + v_{N-n}\otimes \o{v_{N-n}}\right).$$
Finally identifying $v_n\otimes \overline{v_n}$ with the $N\times N$ matrix with a $1$ in the $n^{\text{th}}$ diagonal entry and zeroes elsewhere\footnote{If $V$ is an $N$-dimensional vector space with orthonormal basis $v_0, \ldots , v_{N-1}$
    and if $v = \sum_{i=0}^{N-1} a_i v_i$ then the conjugate symmetric tensor
    $v \otimes \overline{v}$ can be represented by $N \times N$ Gram matrix whose $(i,j)$ entry is $a_i \overline{a_j}$.}
equation (2.10) from Section 2.3 of \cite{bendory2022sample} to write the second moment as a diagonal matrix where 
$$m^2_z[n,n] = \begin{cases}
|z[0]|^2 & n =0, N/2 \\
|z[n]|^2 + |z[N-n]|^2 & n \neq 0, N/2.
\end{cases}$$
Thus the second moment is determined completely by the following quadratic function 
$$b_z = (|z[0]|^2, |z[1]|^2 + |z[N-1]|^2, \dots, |z[N/2]|^2).$$
Our discussion shows the following
equivalence for vectors $x,y \in \C^N$,
$$|Fx|^2 = |Fy|^2 \iff a_x = a_y \iff m^2_x = m^2_y \iff m^2_{Fx} = m^2_{Fy} \iff b_{Fx} = b_{Fy}.$$ 
\end{proof}

We now restrict ourselves to the real subspace  $V = \R^N$
where the original action of $D_{2N}$ is given by 
$$(r\cdot x)[n] = x[n+1] \qquad \text{ and } \qquad (s \cdot x)[n]= x[N-n]$$
and, once again, the indices are taken modulo $N$.
If we consider the  action of $D_{2N}$ with respect to the real
Fourier basis
$$v_\ell = 
\begin{cases}
\sum_{k=0}^{N-1} \cos\left(\frac{2\pi l k}{N}\right)e_k \qquad 0 \leq l \leq \frac{N}{2}\\\\
\sum_{k=0}^{N-1} \sin\left(\frac{2\pi l k}{N} \right)e_k \qquad \frac{N}{2} < l < N\\
\end{cases}$$ 
the vector space $\R^N$ decomposes as a direct sum of irreducible
representations
$$\R^N = V_0 \oplus V_1 \oplus \dots V_{N/2-1} \oplus V_{N/2}$$
where
$V_0$ and $V_{N/2}$ are one dimensional irreducible representations
spanned by $v_0$ and $v_{N/2}$ respectively, and the other $V_k$ are
two-dimensional irreducible representations spanned by $\{v_k,
v_{N-k}\}$, as before.
The action of $D_{2N}$ is as follows:
On the one dimensional
components, we have
\begin{equation*}
	\begin{split}
		r\cdot v_0 &= v_0 \qquad \qquad \, \, \, \, \, s\cdot v_0 = v_0 \\
		r\cdot v_{N/2} &= -v_{N/2} \qquad s\cdot v_{N/2} = v_{N/2}
	\end{split}
\end{equation*}  
and over the two dimensional components, we have that $r$ acts via rotation matrices and again $s$ acts by reflection. In particular, 
\begin{align*}
		r\cdot v_k &= \cos\left(\frac{2\pi k}{N} \right)v_k + \sin\left(\frac{2\pi k}{N}\right)v_{N-k} & s\cdot v_k &= v_{N-k} \\
		r\cdot v_{N-k} &= \cos\left(\frac{2\pi (N-k)}{N} \right)v_k + \sin\left(\frac{2\pi (N-k)}{N}\right)v_{N-k} & s\cdot v_{N-k} &= v_{k}.
\end{align*}  
If we compute the second moment of a vector $z = \sum z[\ell]v_\ell$
we see that it is determined
by the real quadratic function
\begin{equation}
b_z = (z[0]^2, z[1]^2 + z[N-1]^2, \dots, z[N/2]^2).
\end{equation}
We conclude that
$$a_x = a_y \iff b_{F^\texttt{r}x} = b_{F^\texttt{r}y}$$
where $F^{\texttt{r}}x$ is the real discrete Fourier transform
given by the formula
\begin{equation}\label{realfouriereq}
(F^rx)[n] = 
    \begin{cases}
		\sum_{k=0}^{N-1} x[k]\cos\left(\frac{2\pi n k}{N}\right) \qquad \text{ for }  0\leq n \leq \frac{N}{2}\\\\	
		\sum_{k=0}^{N-1} x[k]\sin\left(\frac{2\pi(N/2-n) k}{N}\right) \qquad  \text{ for }  \frac{N}{2} +1 \leq n \leq \frac{N-1}{2} -1.
		\end{cases}
\end{equation}

\subsection{Generic recovery for real signals}
Given an $N$-element frame  $\UU = (\alpha_0, \ldots , \alpha_{N-1})$ 
frame on $\R^N$ we
have that
$b_{L_{\UU}(x)} = [b_{0}, b_{1}, \dots, b_{N/2}]^T$
where 
\begin{equation}
\begin{split}
b_{0} &= \alpha_0^T xx^T\alpha_0 \\
b_{N/2} &= \alpha_{N/2}^Txx^T\alpha_{N/2} \text{ and }\\
b_{k} &= \alpha_k^T xx^T\alpha_k + \alpha_{N-k}^T xx^T \alpha_{N-k} ,\; k=1, \ldots ,
N/2-1 .
\end{split}
\end{equation}
(Note that when $N$ is odd the form of $b_{L_{\UU}(x)}$ is
modified as follows: 
The second equation is removed, and
we replace $N/2$ with $\lfloor N/2 \rfloor$ in
the third equation.)

In the light of
the discussion above and the fact $F^{\texttt{r}}$ is an automorphism
of $\R^N$, Theorem \ref{thm.realsparse}(i) follows from the following
two statements.
\begin{enumerate}
\item[(F1)] If $N > 4M-6$ and $\UU = (\alpha_0, \ldots , \alpha_{N-1})$  is a
  generic frame on $\R^M$ then $b_{L_{\UU}(v)} = b_{L_{\UU}(w)}$ if and only
  if $w = \pm v$.
\item[(F2)] If $N > 4M-4$ and $\UU = (\alpha_0, \ldots , \alpha_{N-1})$
  and $\VV = (\beta_1, \ldots , \beta_N)$ are generic overlapping frames
$b_{L_{\UU}(v)} \neq b_{L_{\UU}(w)}$ unless $L_{\UU}(v) = \pm L_{\VV}(w)$.
\end{enumerate}

Likewise, Theorem \ref{thm.realsparse}(ii) follows from
the following statment.
\begin{enumerate}
\item[(Fg)] If $N > 2M-2$ and $\UU,\VV$ are generic overlapping
  frames then for a generic vector $x \in \R^M$,
  $b_{L_{\UU}(x)} = b_{L_{\VV}}(y)$ if and only
  if $L_{\UU}(x) = \pm L_{\VV}(y)$.
\end{enumerate}

\bigskip

Consider the map
$$\B_{\UU}: \frac{\R^M}{\{\pm 1\}} \to \R^{N/2 + 1},\B_{\UU}(x) = b_{L_{\UU}(x)}.$$
To prove statement (F1) we must show that if $N > 4M-6$ and $\UU=(\alpha_0, \ldots , \alpha_{N-1})$
is a generic frame on $\R^M$ then\\
\centerline{(*) $\B_{\UU}(x) = \B_{\UU}(y)$ if and only if $x=y$.}

Adapting the proof of \cite[Lemma 9]{bandeira2014saving}
we obtain the following:
\begin{proposition}\label{real_lemma9}
Condition (*) fails if and only if there is a non-zero real symmetric matrix $ A $ such that $\rank(A) \leq 2$ and
\begin{equation} \label{eq.dagger}
  \begin{array}{l}
     \alpha_0^T A \alpha_0  = 0\\
\alpha_{N/2}^T A \alpha_{N/2}  =  0\\
\alpha_k^T A \alpha_{k} + \alpha_{N-k}^T A \alpha_{N-k}  =  0, k = 1,\dots, \frac{N}{2}-1
  \end{array}
  \end{equation}
\end{proposition}
We now follow the same strategy of \cite{conca2015algebraic}
to show that if $N > 4M-6$ then, for a generic frame $\UU$,
there are no rank-two symmetric matrices satisfying \eqref{eq.dagger}.
To that end let $V^2_M$ be the affine
variety 
of symmetric matrices of rank at most 2
and
consider the projective incidence variety
$$\b_{M,N} \subseteq \Pro\left((\C^{M})^N\right) \times
\Pro(V^2_M)$$ where 
$([\alpha_0, \ldots
, \alpha_{N-1}], [A])$ of $\b_{M,N,s}$ satisfies equations \eqref{eq.dagger}.

The analogue of Proposition \ref{injectiveiff} for
this case is the following.
\begin{proposition}\label{real_injectiveiff}
  Given the set up above, condition (*) holds for the frame $\UU = (\alpha_0, \ldots , \alpha_{N-1})$
  if and only if
  $(\alpha_0, \ldots, \alpha_{N-1}) \notin \pi_1((\b_{M,N,s})_\R)$
  \end{proposition}

\begin{proposition}\label{real_dimbmn}
The dimension of $\b_{M,N}$ equals $2M + MN-N/2-4$.
\end{proposition}

In order to compute the dimension of $\b_{M,N}$ we need to know 
the dimension of $V_M^2$, the locus of complex symmetric matrices of rank at most two.
The following result on the dimensions of loci of symmetric matrices of a given rank
is well known and 
an outline of a proof
is given in \cite[Chapter II, Exercises A1-A2]{acgh}.
\begin{lemma}\label{symdim}
Let $k\leq M-2$. The locus  $V_{M}^k$ of complex symmetric matrices with rank $\leq k$ forms an affine algebraic variety with dimension
$\displaystyle{\dim V_{M}^k = k + {M \choose 2} - {M-k \choose 2}}.$
\end{lemma}
In particular, we see that $\dim V^2_M = 2M-1$. We also note for later use
that 
$\dim V^1_M = M$. 
Much like in the proof of Proposition \ref{dimbmn}, our strategy to
compute the dimension of $\b_{M,N,s}$ is to show that $\pi_2$ is surjective
and compute the dimension of the fibers $\pi_2^{-1}([A])$ for $A  \in \C^{M\times M}_{\mathrm{sym}}$. 
\begin{lemma}\label{real_claim_projection}
  The projection
  $\pi_2 \colon \B_{M,N} \to \Pro(V^2_M)$
  is surjective.
\end{lemma}
\begin{proof}
For any $M\times M$ complex symmetric matrix $A$ with rank $\leq 2$,
pick some $\alpha \in \C^M$ such that $\alpha^T A \alpha = 0$. This
can be done since $u^T A u = 0$ is a non-degenerate homogeneous
quadratic equation in the entries of $u$, and the relation thus
defines a quadric hypersurface. Any point in this hypersurface
gives a choice of $\alpha$. Let $\alpha_k=\alpha$ for all $k$,
which gives us the desired point in $\b_{M,N,s}$ whose is image
under $\pi_2$ is $[A]$.
\end{proof}

\begin{proof}[Proof of Proposition \ref{real_dimbmn}]
From Lemmas \ref{symdim}, \ref{real_claim_projection}, we see that
$\dim(\pi_2(\b_{M,N,s})) = 2M-2$. Now, we focus on the fibers. As in the proof
Lemma \ref{dim}, we notice that for any fixed symmetric matrix $A$, the equations defined
by the system \eqref{eq.dagger} are all non-zero and algebraically
independent since  no
two equations are defined by polynomials in the same
variables. The first two equations are non-zero quadratic
equations in $M$ variables, and thus define 
quadric hypersurfaces in $\C^M$, while the other equations each define quadric hypersurfaces in $\C^{2M}$. After projectivizing
we see that 
$$\dim(\pi_2^{-1}([A])) =
2(M-1)+\left(\frac{N}{2} -1\right)(2M-1) - 1 = MN - \frac{N}{2}-2$$
and that each fiber is a product of hypersurfaces.
This implies that the map
$\pi_2 \colon \B_{M,N} \to \Pro(V^2_M)$ is flat
and we conclude that every irreducible component of
$\B_{M,N}$ has dimension $2M-2 + MN - N/2 -2$.
\end{proof}

\begin{theorem}\label{real_genericL}
For a generic frame $\UU=(\alpha_0, \ldots, \alpha_{N-1})$ condition (*) holds when $N>4M-6$
\end{theorem}
\begin{proof}
We know from Proposition \ref{real_injectiveiff} that condition (*)
holds if and only if
$$(\alpha_0, \ldots, \alpha_{N-1}) \not\in
\pi_1((\b_{M,N,s})_\R) \subset \pi_1(\b_{M,N,s})_\R.$$
Since the
dimension of the image of a projection is at most the dimension of the
source, we see that $\dim(\pi_1(\b_{M,N,s})) \leq 2M + MN -N/2 -4$.
Finally, we see that when $N>4M-6$, $$2M + MN- \frac{N}{2}-4 < MN-1 =
\dim \Pro\left((\C^M)^N\right).$$ Thus, the points of $\b_{M,N,s}$ for
which condition (*) fails lie in a proper algebraic subset.
\end{proof}

This proves (F1) and we now proceed to prove (F2) and
the corresponding statements (Fg).
Our proof is similar to the proof we gave in Section\ref{complexcase}, so
we refrain from repeating a lot of
arguments.

Suppose that $\UU = (\alpha_0,\ldots \alpha_{N-1})$ and  $\VV =(\beta_0, \ldots , \beta_{N-1})$ are
two $s$-overlapping frames on $\R^N$
and consider the map 
 $$ \B_{\UU\VV}: (\R^M/\pm 1) \times (\R^M/\pm 1) \to \R^N,
    (x,y) \mapsto b_{L_{\UU}(x)} - b_{L_{\VV}}(y)$$
Our goal is to show is to show that if $N > 4M-4$ then
for a generic pair of $s$-overlapping frame $\UU, \VV$\\
\centerline{(**) $\mathcal{A}_{\UU \VV}(x,y) = 0$ if and only if $x =y$.}

As before, we can extend $\B_{\UU\VV}$ to linear map
 $\B_{\UU \VV} \colon \R^{M \times M}_{\textrm{sym}} \times \R^{M \times M}_{\textrm{sym}} \to \R^N$ given by
\begin{gather*}  (A, B) \mapsto (\alpha_0^TA\alpha_0 - \beta_0^T B \beta_0,
  \alpha_{N/2}^T A \alpha_{N/2} - \beta_{{N\over{2}}}^T B \beta_{{N\over{2}}},
  \alpha_1^TA \alpha_1 + \alpha_{N-1}^T A \alpha_{N-1} - \beta_1^TB \beta_1
  \beta_{N-1}^T B \beta_{N-1},\\ \ldots , \alpha_{{N\over{2}}-1}^TA \alpha_{{N\over{2}}-1} +
  \alpha_{{N\over{2}}+1}^T A \alpha_{{N\over{2}}+1} - \beta_{{N\over{2}}+1}^TB \beta_{{N\over{2}}+1}
    -    \beta_{N/2+1}^T B \beta_{N/2+1})
\end{gather*}
Let ${\mathbb U}$ be the quasi-affine subvariety of $\R^{M \times M}_{\textrm{sym}} \times \R^{M \times M}_{\textrm{sym}}$ parametrizing
    pairs $(A,B)$ of distinct rank-one symmetric matrices. As before we observe that the following holds.
\begin{proposition}\label{injcondition}
Condition (**) holds if and only if ${\mathbb U}$ does not intersect
  the kernel of the linear map $\mathcal{A}_{\UU \VV} \colon \R^{M \times M}_{\textrm{sym}} \times \R^{M \times M}_{\textrm{sym}} \to \R^N$.
\end{proposition}  
Thus, to prove statement (F2) and complete the proof of Theorem \ref{thm.realsparse}(i)
we must prove the following result.
\begin{theorem}\label{real_genericpairs}
If $N > 4M-4$ and $\UU$ and $\VV$ are generic frames $s$-overlapping frames then the set ${\mathbb U}$ does not intersect the kernel
  of the linear map $\mathcal{B}_{\UU\VV} \colon \R^{M \times M}_{\textrm{sym}} \times \R^{M \times M}_{\textrm{sym}} \to \R^N$
\end{theorem}
Likewise, statement (Fg) or, equivalently, Theorem \ref{thm.realsparse}(ii) follows
from the following result.
\begin{theorem}\label{real_generica}
  Suppose $N > 2M-2$ and $\UU,\VV$ are generic pair of $s$-overlapping frames.
Then if $A \in V^1_M$ is generic there is no 
$B \neq A \in V^1_M$ such that
$\B_{\UU\VV}(A,B) = 0$.
\end{theorem}

Define
$$\b_{M,N,s} \subseteq \Pro(\Lambda_s) \times \Pro(V_M^1 \times V_M^1)$$ as the set of all tuples
$$([\alpha_0, \ldots , \alpha_{N-1}, \beta_0, \ldots , \beta_{N-1}],  [A ,  B] )$$
satisfying
\begin{equation} \label{eq.secondtime}
\begin{array}{l}
  \alpha_0^T A \alpha_0  = \beta_0^T B \beta_0\\
\alpha_{N/2}^T A \alpha_{N/2}  = \beta_{N/2}^T B \beta_{N/2}\\
\alpha_k^T A \alpha_{k} + \alpha_{N-k}^T A \alpha_{N-k}  = \beta_k^T B \beta_{k} + \beta_{N-k}^T B \beta_{N-k}\;
 k = 1,\dots, \frac{N}{2}-1
\end{array}
\end{equation}
Given our definition, 
the kernel of the linear map $\B_{\UU\VV}$ does not contain any distinct pairs of rank-one symmetric matrices if and only if
$[\alpha_0,  \dots, \alpha_N, \beta_0, \ldots , \beta_N] \not\in \pi_1((\b_{M,N,s})_\R)$

\begin{lemma}\label{claim_real_pairs}
  The projection $\pi_2$ is surjective; i.e. every pair $[( A ,  B )] \subseteq
  \Pro(V_M^1 \times V_M^1)$ with $ A \neq  B $ belongs to the projection $\pi_2(\b_{M,N,s})$. 
\end{lemma}
\begin{proof}
  If $s < M$, pick a pair $(A ,B) \subseteq V_M^1 \times V_M^1$
  and choose non-zero vectors $\alpha , \beta \in R^M$ whose first $s$-coordinates
  are zero and satisfies the equation
  $\alpha A \alpha  - \beta^T B \beta =0$. Now take
  $\UU = (\alpha, \ldots , \alpha)$ and $\VV = (\beta, \ldots , \beta)$.
  If $s=M$ choose a vector
  satisyfing the equation $\alpha^T( A  -  B )\alpha = 0$
  and let $\UU = \VV = (\alpha,\ldots , \alpha)$
  In either case $([\UU, \VV] , [A, B]) \in \B_{M,N,s}$ so $[A,B] \in \pi_2(\B_{M,N,s})$.
 
\end{proof}

\begin{proposition}
The dimension of $\b_{M,N,s}$ equals $2MN + 2M - \frac{N}{2}-sN-3$
\end{proposition}
\begin{proof}
  By Lemma \ref{claim_real_pairs} we know that $\pi_2$ is surjective
  so to compute the dimension of $\b_{M,N,s}$ we just need
  to compute the dimension of the fibers of $\pi_2$.
  We use the same technique as in the proof of Lemma \ref{dim}
  to verify that for a given pair $( A ,  B ) \in V_M^1 \times V_M^1$ with $ A  \neq  B $ the equations that make up
\eqref{eq.secondtime}
are all non-degenerate and are algebraically independent (the latter is true since each equation involves a different set of indeterminates).
This implies that the fiber
over any $(A,B)$ with $A \neq B$ is the product of $N/2 +1$ quadric hypersurfaces
whose dimensions are $2M-s-1, 2M-s-1, \overbrace{4M-1-2s,\ldots , 4M-1-2s}^{N/2 -1}$
respectively.
Putting this together we have that the dimension of the fibers, after projectivization is
	$$\dim \pi_2^{-1}( [A , B] ) = 2(2M-1-s)+ \left(\frac{N}{2}-1\right)(4M-1-2s) - 1 = 2MN - \frac{N}{2}-sN-2$$
Once again this implies that  map $\pi_2$ is flat and we conclude that
$$\dim(\b_{M,N,s})= 2M -1 + 2MN - \frac{N}{2}-sN-2 = 2M + 2MN -\frac{N}{2} -sN -
3$$
\end{proof}
\begin{proof}[Proof of Theorem \ref{real_genericpairs}]
We know that the linear map $\B_{\UU,\VV}$ does not contain any distinct pairs of rank-one symmetric matrices if and only if
$[\alpha_0,  \dots, \alpha_N, \beta_0, \ldots , \beta_N] \not\in \pi_1((\b_{M,N,s})_\R)$
Since the dimension of the image of a projection is at most the dimension of the source, we see that
$$\dim(\pi_1(\b_{M,N,s})) \leq \dim(\b_{M,N,s}) =2MN + 2M - \frac{N}{2}-sN-3$$
When $N>4M-4$, we have $\dim(\pi_1(\b_{M,N,s})) < 2MN- sN - 1 = \dim \Pro(\Lambda_s)$, and so
the generic pair of $s$-overlapping frames $\UU = (\alpha_0,\ldots, \alpha_{N-1}), \VV = (\beta_0, \ldots , \beta_{N-1})$ does not lie in $\pi_1((\b_{M,N,s})_\R)$
\end{proof}

\begin{proof}[Proof of Theorem \ref{real_generica}]
  Consider the morphism $\pi_1: \b_{M,N,s}\to \Pro(\Lambda_s)$.
  Arguing exactly as in the first paragraph of the proof of Theorem \ref{generica}
  we can reduce to the case that $\pi_1$ is dominant. Once again
we can apply 
Proposition \ref{prop.vakil} and conclude that there is a Zariski open set of
points $q = [\UU,\VV]$ such that the dimension of the fiber
  \begin{eqnarray*}
    \dim \pi_1^{-1}(q) & = & \dim \b_{M,N,s} - \dim \Pro(\Lambda_s)\\
    &= & \left(2MN + 2M - \frac{N}{2}-sN-3\right) - (2MN-sN-1) = 2M - \frac{N}{2}-2
    \end{eqnarray*}
with the understanding that if $N$ is such that $2M - N/2 -2 < 0$ then the fiber is empty.

Much like in the proof of Theorem \ref{generica} the fiber of $\pi_1$ over $[(\UU,\VV)] \in \Pro(\Lambda_s)$
is the quasi-projective subvariety $\Gamma$ of $\Pro(V_M^1 \times V_M^1)$ given by
$$\Gamma = \{[A ,B] \in  \Pro(V_M^1 \times V_M^1)|  \eqref{eq.secondtime} \text{ holds for } [A,B] \}$$
and consider the projection $p_1$ onto the first coordinate, which is well-defined because
we assume $A \neq 0$. Since
$\dim \overline{p_1(\Gamma)} \leq \dim \Gamma = 2M - \frac{N}{2}-
2.$
we see that when  $N>2M-2$, $\pi_1(\Gamma)$ lies in a proper algebraic subset of
        $\Pro(V_M^1)$.
This implies that if $N > 2M-2$, for a generic choice of $s$-overlapping frames $(\UU,\VV)$
the set of rank-one $A$ such that there is no rank-one $B$ with $A \neq B$ satisfying
\eqref{eq.secondtime} is Zariski dense.
 Restricting 
$[U,U',V,V']$ to the real points of $\Pro(\Lambda_s)$ corresponding to $s$-overlapping real frames
frames and restricting $A$ to the real points $V^1_M$ corresponding
        to rank-one symmetric real matrices
        we see that for a 
        a generic pair of $N > 2M-2$ element overlapping frames $\UU, \VV$ and a
        generic choice of vector $x \in \R^M$, there is no $y \neq \pm x$
such that $b_{L_{\UU}(x)} = b_{L_{\VV}(y)}$.
\end{proof}
\section{Future directions} Our proof of Theorem~\ref{thm.realsparse} does not give an explicit method for determining whether a specific basis is generic, nor does it provide an algorithm for reconstruction. A recent paper~\cite{amir2026stability} gives a criterion for the bi-Lipschitz stability of the second moment map for signals which are sparse with respect to a generic basis. A natural question for further work is to determine whether the results of~\cite{amir2026stability} can be used to give an explicit algorithm for determining when sparse vectors with respect to given basis can be recovered from their power spectrum. We note that for the related problem of recovering a sparse signal from the absolute value of its continuous Fourier transform, Beinert and Plonka~\cite{beinert2018enforcing} used Prony's method to give an algorithm for signal recovery up to dihedral ambiguities.
\section*{Acknowledgment}
Both authors were supported by the BSF grant no. 2020159 and D.E. was also supported by NSF-DMS 2205626.
\bibliographystyle{plain}
%\bibliography{ref}

\begin{thebibliography}{10}

\bibitem{amir2026stability}
Tal Amir, Tamir Bendory, Nadav Dym, and Dan Edidin.
\newblock The stability of generalized phase retrieval problem over compact
  groups.
\newblock {\em Applied and Computational Harmonic Analysis}, 82:101838, 2026.

\bibitem{acgh}
E.~Arbarello, M.~Cornalba, P.~A. Griffiths, and J.~Harris.
\newblock {\em Geometry of algebraic curves. {V}ol. {I}}, volume 267 of {\em
  Grundlehren der mathematischen Wissenschaften [Fundamental Principles of
  Mathematical Sciences]}.
\newblock Springer-Verlag, New York, 1985.

\bibitem{balan2009painless}
Radu Balan, Bernhard~G. Bodmann, Peter~G. Casazza, and Dan Edidin.
\newblock Painless reconstruction from magnitudes of frame coefficients.
\newblock {\em J. Fourier Anal. Appl.}, 15(4):488--501, 2009.

\bibitem{balan2006signal}
Radu Balan, Pete Casazza, and Dan Edidin.
\newblock On signal reconstruction without phase.
\newblock {\em Applied and Computational Harmonic Analysis}, 20(3):345--356,
  2006.

\bibitem{bandeira2014saving}
Afonso~S. Bandeira, Jameson Cahill, Dustin~G. Mixon, and Aaron~A. Nelson.
\newblock Saving phase: injectivity and stability for phase retrieval.
\newblock {\em Appl. Comput. Harmon. Anal.}, 37(1):106--125, 2014.

\bibitem{beinert2015ambiguities}
Robert Beinert and Gerlind Plonka.
\newblock Ambiguities in one-dimensional discrete phase retrieval from
  {Fourier} magnitudes.
\newblock {\em Journal of Fourier Analysis and Applications}, 21(6):1169--1198,
  2015.

\bibitem{beinert2018enforcing}
Robert Beinert and Gerlind Plonka.
\newblock Enforcing uniqueness in one-dimensional phase retrieval by additional
  signal information in time domain.
\newblock {\em Applied and Computational Harmonic Analysis}, 45(3):505--525,
  2018.

\bibitem{bendory2023autocorrelation}
T.~Bendory, Y.~Khoo, J.~Kileel, and A.~Singer.
\newblock Autocorrelation analysis for cryo-{EM} with sparsity constrants:
  Improved sample complexity and projection-based algorithms.
\newblock {\em Proceedings of the National Academy of Sciences}, 2023.

\bibitem{bendory2020toward}
Tamir Bendory and Dan Edidin.
\newblock Toward a mathematical theory of the crystallographic phase retrieval
  problem.
\newblock {\em SIAM Journal on Mathematics of Data Science}, 2(3):809--839,
  2020.

\bibitem{bendory2022sample}
Tamir Bendory and Dan Edidin.
\newblock The sample complexity of sparse multireference alignment and
  single-particle cryo-electron microscopy.
\newblock {\em SIAM J. Math. Data Sci.}, 6(2):254--282, 2024.

\bibitem{bendory2023finite}
Tamir Bendory, Dan Edidin, and Ivan Gonzalez.
\newblock Finite alphabet phase retrieval.
\newblock {\em Applied and Computational Harmonic Analysis}, 66:151--160, 2023.

\bibitem{bodmann2015stable}
Bernhard~G. Bodmann and Nathaniel Hammen.
\newblock Stable phase retrieval with low-redundancy frames.
\newblock {\em Adv. Comput. Math.}, 41(2):317--331, 2015.

\bibitem{bajarovska2016phase}
Irena Bojarovska and Axel Flinth.
\newblock Phase retrieval from {G}abor measurements.
\newblock {\em J. Fourier Anal. Appl.}, 22(3):542--567, 2016.

\bibitem{botelho-andrade2016phase}
Sara Botelho-Andrade, Peter~G. Casazza, Hanh Van~Nguyen, and Janet~C. Tremain.
\newblock Phase retrieval versus phaseless reconstruction.
\newblock {\em J. Math. Anal. Appl.}, 436(1):131--137, 2016.

\bibitem{candes2013phase}
Emmanuel~J. Cand\`es, Yonina~C. Eldar, Thomas Strohmer, and Vladislav
  Voroninski.
\newblock Phase retrieval via matrix completion.
\newblock {\em SIAM J. Imaging Sci.}, 6(1):199--225, 2013.

\bibitem{candes2013phaselift}
Emmanuel~J. Cand\`es, Thomas Strohmer, and Vladislav Voroninski.
\newblock Phase{L}ift: exact and stable signal recovery from magnitude
  measurements via convex programming.
\newblock {\em Comm. Pure Appl. Math.}, 66(8):1241--1274, 2013.

\bibitem{conca2015algebraic}
Aldo Conca, Dan Edidin, Milena Hering, and Cynthia Vinzant.
\newblock An algebraic characterization of injectivity in phase retrieval.
\newblock {\em Appl. Comput. Harmon. Anal.}, 38(2):346--356, 2015.

\bibitem{daubechies1992ten}
Ingrid Daubechies.
\newblock {\em Ten lectures on wavelets}, volume~61 of {\em CBMS-NSF Regional
  Conference Series in Applied Mathematics}.
\newblock Society for Industrial and Applied Mathematics (SIAM), Philadelphia,
  PA, 1992.

\bibitem{elser2017complexity}
Veit Elser.
\newblock The complexity of bit retrieval.
\newblock {\em IEEE Transactions on Information Theory}, 64(1):412--428, 2017.

\bibitem{elser2018benchmark}
Veit Elser, Ti-Yen Lan, and Tamir Bendory.
\newblock Benchmark problems for phase retrieval.
\newblock {\em SIAM Journal on Imaging Sciences}, 11(4):2429--2455, 2018.

\bibitem{evans2020conjugate}
Luke Evans and Chun-Kit Lai.
\newblock Conjugate phase retrieval on {$\Bbb{C}^M$} by real vectors.
\newblock {\em Linear Algebra Appl.}, 587:45--69, 2020.

\bibitem{fickus2014phase}
Matthew Fickus, Dustin~G. Mixon, Aaron~A. Nelson, and Yang Wang.
\newblock Phase retrieval from very few measurements.
\newblock {\em Linear Algebra Appl.}, 449:475--499, 2014.

\bibitem{ghosh2022sparse}
Subhroshekhar Ghosh and Philippe Rigollet.
\newblock Sparse multi-reference alignment: Phase retrieval, uniform
  uncertainty principles and the beltway problem.
\newblock {\em Foundations of Computational Mathematics}, pages 1--48, 2022.

\bibitem{Harris}
Joe Harris.
\newblock {\em Algebraic geometry}, volume 133 of {\em Graduate Texts in
  Mathematics}.
\newblock Springer-Verlag, New York, 1995.
\newblock A first course, Corrected reprint of the 1992 original.

\bibitem{hartshorne}
Robin Hartshorne.
\newblock {\em Algebraic geometry}.
\newblock Graduate Texts in Mathematics, No. 52. Springer-Verlag, New
  York-Heidelberg, 1977.

\bibitem{iwen2017robust}
Mark Iwen, Aditya Viswanathan, and Yang Wang.
\newblock Robust sparse phase retrieval made easy.
\newblock {\em Appl. Comput. Harmon. Anal.}, 42(1):135--142, 2017.

\bibitem{li2019phase}
Lan Li, Ted Juste, Joseph Brennan, Chuangxun Cheng, and Deguang Han.
\newblock Phase retrievable projective representation frames for finite abelian
  groups.
\newblock {\em J. Fourier Anal. Appl.}, 25(1):86--100, 2019.

\bibitem{patterson1934fourier}
A.~L. Patterson.
\newblock A {F}ourier series method for the determination of the components of
  interatomic distances in crystals.
\newblock {\em Phys. Rev.}, 46:372--376, Sep 1934.

\bibitem{patterson1944ambiguities}
A.~L. Patterson.
\newblock Ambiguities in the {X}-ray analysis of crystal structures.
\newblock {\em Phys. Rev.}, 65:195--201, Mar 1944.

\bibitem{wang2014phase}
Yang Wang and Zhiqiang Xu.
\newblock Phase retrieval for sparse signals.
\newblock {\em Appl. Comput. Harmon. Anal.}, 37(3):531--544, 2014.

\end{thebibliography}

\end{document}